\documentclass[runningheads, envcountsame, a4paper]{llncs}

\usepackage{etex}
\usepackage{amssymb,amsmath}
\usepackage{latexsym}
\usepackage[utf8]{inputenc}
\usepackage[T1]{fontenc}
\usepackage{lmodern}         
\usepackage[all]{xy}
\usepackage{graphicx}
\usepackage{array}
\usepackage{color}
\usepackage{subfigure}
\usepackage{hyperref}
\usepackage{enumerate}
\usepackage{todonotes}
\usepackage{tabularx}
\usepackage{multirow}
\usepackage{tikz}

\usepackage{url}

\usepackage{hyperref}

\DeclareMathOperator{\fac}{Fac}
\DeclareMathOperator{\card}{Card}

\author{J. Cassaigne \and S. Labbé \and J. Leroy\thanks{J. Leroy is FNRS post-doctoral fellow.}}
\title{A Set of Sequences of Complexity $2n+1$}
\institute{
Institut de mathématiques de Marseille, CNRS UMR 7373, Campus de Luminy,
Case 907, 13288 Marseille Cedex 09, France\\
\email{julien.cassaigne@math.cnrs.fr}\\
CNRS, LaBRI, UMR 5800, F-33400 Talence, France\\
\email{sebastien.labbe@labri.fr}\\
Universit\'e de Li\`ege, Institut de math\'ematique, \\
All\'ee de la d\'ecouverte 12 (B37), 4000 Li\`ege, Belgium\\
\email{j.leroy@ulg.ac.be}
}

\toctitle{A Set of Sequences of Complexity $2n+1$}
\tocauthor{Julien~Cassaigne, Sébastien~Labbé, and Julien~Leroy} 

\begin{document}

\maketitle
\setcounter{footnote}{0}

\newcommand{\N}{\mathbb{N}}
\newcommand{\Z}{\mathbb{Z}}
\newcommand{\Q}{\mathbb{Q}}
\newcommand{\R}{\mathbb{R}}
\newcommand{\bx}{\mathbf{x}}
\newcommand{\A}{\mathcal{A}}
\newcommand{\MONE}{{\arraycolsep=2pt\left(\begin{array}{rrr}
1 & 1 & 0 \\
0 & 0 & 1 \\
0 & 1 & 0
\end{array}\right)}}
\newcommand{\MTWO}{{\arraycolsep=2pt\left(\begin{array}{rrr}
0 & 1 & 0 \\
1 & 0 & 0 \\
0 & 1 & 1
\end{array}\right)}}

\def\bw{\mathbf{w}}
\def\bs{\mathbf{s}}
\def\ba{\mathbf{a}}

\begin{abstract}
We prove the existence of a ternary sequence of factor complexity $2n+1$ for
any given vector of rationally independent letter frequencies. Such sequences
are constructed from an infinite product of two substitutions according to a
particular Multidimensional Continued Fraction algorithm. We show that this
algorithm is conjugate to a well-known one, the Selmer algorithm. 
Experimentations (Baldwin, 1992) suggest that their second Lyapunov exponent is
negative which presages finite balance properties.

\end{abstract}

\section{Introduction}

Words of complexity $2n+1$ were considered in \cite{MR1116845}
with the condition that there is exactly one left and one right special factor of
each length. These words are called Arnoux-Rauzy sequences and are a generalization
of Sturmian sequences on a ternary alphabet.
It is known that the frequencies
of any Arnoux-Rauzy word are well defined and belong to the Rauzy Gasket
\cite{MR3184185}, a fractal set of Lebesgue measure zero.
Thus the above condition on the number of special factors is very restrictive for the possible letter frequencies.

Sequences of complexity $p(n)\leq 2n+1$ include Arnoux-Rauzy words, 
codings of interval exchange transformations and more \cite{MR3214265}. 
For any given letter frequencies one can construct sequences of factor
complexity $2n+1$ by the coding of a 3-interval exchange transformation. It
is known that these sequences are unbalanced \cite{MR1488330}.
Thus the question of finding balanced ternary sequences of factor complexity
$2n+1$ for all letter frequencies remains. This article intends to give a positive answer to this question for almost all vectors of letter frequencies (with respect to Lebesgue measure).

In recent years, multidimensional continued fraction algorithms were used to
obtain ternary balanced sequences with low factor complexity for any given
letter frequency vector.  Indeed the Brun algorithm leads to balanced sequences
\cite{MR3124516} and it was shown that the Arnoux-Rauzy-Poincaré
algorithm leads to sequences of factor complexity $p(n)\leq\frac{5}{2}n+1$
\cite{MR3283831}.

In 2015, the first author introduced a new Multidimensional Continued Fraction
algorithm \cite{cassaigne_algorithme_2015} based on the study of Rauzy graphs.
In this work, we formalize the algorithm,
its matrices, substitutions and associated cocycles and $S$-adic words. 
We show that $S$-adic words obtained from these substitutions
have complexity $2n+1$. We also show that the algorithm is conjugate to the
Selmer algorithm, a well-known Multidimensional Continued Fraction algorithm.
We believe that almost all sequences generated by the algorithm are balanced.

\section{A Bidimensional Continued Fraction Algorithm}
On $\Lambda= \mathbb{R}^3_{\geq 0}$, the bidimensional continued fraction algorithm
introduced by the first author~\cite{cassaigne_algorithme_2015} is
\[
F_C (x_1,x_2,x_3) = 
\begin{cases}
    (x_1-x_3, x_3, x_2), & \mbox{if } x_1 \geq x_3;\\
    (x_2, x_1, x_3-x_1), & \mbox{if } x_1 < x_3.
\end{cases}
\]
More information on Multidimensional Continued Fraction Algorithms
can be found in \cite{BRENTJES,schweiger}.

\subsection{The Matrices}
Alternatively, the map $F_C$ can be defined by associating nonnegative matrices
to each part of a partition of $\Lambda$ into $\Lambda_1\cup\Lambda_2$ where
\begin{align*}
	\Lambda_1 &= \{(x_1,x_2,x_3)\in\Lambda\mid 
	x_1 \geq x_3\}, \\
    \Lambda_2 &= \{(x_1,x_2,x_3)\in\Lambda\mid 
	x_1 < x_3\}.
\end{align*}
The matrices are given by the rule
$M(\bx) = C_i$
if and only if
$\bx\in\Lambda_i$ where
\[
C_{1}=\MONE
\qquad\text{and}\qquad
C_{2}=\MTWO.
\]
The map $F_C$ on $\Lambda$ and
the projective map $f_C$ on
$\Delta=\{\bx\in\Lambda\mid\Vert\bx\Vert_1=1\}$ are then defined as:
\[
    F_C(\bx) = M(\bx)^{-1}\bx
    \qquad\text{and}\qquad
    f_C(\bx) = \frac{F_C(\bx)}{\Vert F_C(\bx)\Vert_1}.
\]
Many of its properties can be found in \cite{labbe_3-dimensional_2015} and the
density function of the invariant measure of $f_C$ was computed in
\cite{arnoux_labbe_2017}.

\subsection{The Cocycle}
The algorithm $F_C$ defines a \emph{cocycle} $M_n:\Lambda\to SL(3,\Z)$ by
\[
    M_0(\bx) = I 
    \quad
    \text{and}
    \quad
    M_n(\bx) = 
    M(\bx)
    M(F_C\bx)
    M(F_C^2\bx)\cdots
    M(F_C^{n-1}\bx)
\]
satisfying the cocycle property $M_{n+m}(\bx) = M_n(\bx)\cdot M_m(F_C^n\bx)$.

For example starting with $\bx=(1,e,\pi)^T$, the first iterates (approximate to
the nearest hundredth) under $F_C$ are
\[
\left(\begin{array}{r} 1.00\\ 2.72\\ 3.14\end{array}\right)
\xrightarrow{F_C}
\left(\begin{array}{r} 2.72\\ 1.00\\ 2.14\end{array}\right)
\xrightarrow{F_C}
\left(\begin{array}{r} 0.58\\ 2.14\\ 1.00\end{array}\right)
\xrightarrow{F_C}
\left(\begin{array}{r} 2.14\\ 0.58\\ 0.42\end{array}\right)
\xrightarrow{F_C}
\left(\begin{array}{r} 1.72\\ 0.42\\ 0.58\end{array}\right)
\xrightarrow{F_C}
\left(\begin{array}{r} 1.14\\ 0.58\\ 0.42\end{array}\right)
\]
The associated cocycle at $\bx=(1,e,\pi)^T$ when $n=5$ is
\begin{align*}
    M_5(\bx) 
    &= M(\bx) M(F_C\bx) M(F_C^2\bx) M(F_C^3\bx) M(F_C^4\bx) \\
    &= C_{2} C_{1} C_{2} C_{1} C_{1} \\
    &= \MTWO \MONE \MTWO \MONE \MONE 
    = {\arraycolsep=2pt\left(\begin{array}{rrr}
	0 & 1 & 1 \\
	1 & 2 & 1 \\
	1 & 2 & 2
	\end{array}\right)}.
\end{align*}

\subsection{The Substitutions}

Let $\A=\{1,2,3\}$.
The substitutions on $\A^*$ are given by the rule
$\sigma(\bx) = c_i$
if and only if
$\bx\in\Lambda_i$ for $i=1,2$ where
\[
c_{1}=\left\{\begin{array}{l}
1 \mapsto 1\\
2 \mapsto 13\\
3 \mapsto 2
\end{array}\right.
\qquad\text{and}\qquad
c_{2}=\left\{\begin{array}{l}
1 \mapsto 2\\
2 \mapsto 13\\
3 \mapsto 3
\end{array}\right.
\]
One may check that $C_i$ is the incidence matrix of $c_i$ for $i=1,2$.
For any word $w\in\A^*$, we denote $\overrightarrow{w}=(|w|_1,|w|_2,|w|_3)\in\N^3$
where $|w|_i$ means the number of occurrences of the letter $i$ in $w$.
Therefore, for all $\bx\in\Lambda$, $\sigma(\bx):\A^*\to\A^*$ is a 
monoid morphism such that its incidence matrix is $M(\bx)$, i.e.,
$\overrightarrow{\sigma(\bx)(w)} = M(\bx)\cdot\overrightarrow{w}$.

\subsection{$S$-adic Words}

Let $S$ be a set of morphisms.
A word $\bw$ is said to be {\em $S$-adic} if there is a sequence $\bs=(\tau_n:A_{n+1}^* \to A_n^*)_{n \in \N} \in S^\N$ and a sequence $\ba = (a_n) \in \prod_{n\in\N}A_n$ such that $\bw = \lim_{n \to +\infty} \tau_0 \tau_1 \cdots \tau_{n-1}(a_{n})$.
The pair $(\bs,\ba)$ is called an {\em $S$-adic representation} of $\bw$ and the sequence $\bs$ a {\em directive sequence} of $\bw$.
The $S$-adic representation is said to be {\em primitive} whenever the directive sequence $\bs$ is primitive, i.e., for all $r \geq 0$, there exists $r'>r$ such that all letters of $A_r$ occur in all images $\tau_r \tau_{r+1} \cdots \tau_{r'-1}(a)$, $a \in A_{r'}$.
Observe that if $\bw$ has a primitive $S$-adic representation, then $\bw$ is uniformly recurrent.
For all $n$, we set $\bw^{(n)} = \lim_{m \to +\infty} \tau_n \tau_{n+1} \cdots \tau_{m-1}(a_{m})$.

\subsection{$S$-adic Words Associated with the Algorithm $F_C$}

The algorithm $F_C$ defines the function $\sigma_n:\Lambda\to \mathrm{End}(\A^*)$, 
$    \sigma_n(\bx) = 
    \sigma(F_C^n\bx)
$
When the sequence $(\sigma_n(\bx))_{n \in \mathbb{N}}$ contains infinitely many occurrences of $c_1$ and $c_2$, this defines a $\mathcal{C}$-adic word, $\mathcal{C} = \{c_1,c_2\}$,
\[
    W(\bx) = \lim_{n\to\infty} \sigma_0(\bx) \sigma_1(\bx) \cdots \sigma_n(\bx) (1).
\]
Indeed, let $w_n = \sigma_0(\mathbf x) \cdots \sigma_n(\mathbf x)(1)$.
As $c_1$ and $c_2$ occur infinitely often, there exist infinitely many indices $m$ such that $\sigma_{m+1}(\mathbf x) = c_1$ and $\sigma_{m+2}(\mathbf x) = c_2$. 
For all $n\geq m+2$, let $z=\sigma_{m+3}(\mathbf x) \cdots \sigma_n(\mathbf x)(1)$.
Since $\{1,2\}\cal A^*$ is stable under both $c_1$ and $c_2$, we have
$z \in \{1,2\}\cal A^*$, so that $c_1c_2(z) \in \{13,12\}\cal A^*$.
Then $1$ is a proper prefix of $c_1c_2(z)=\sigma_{m+1}(\mathbf x) \cdots \sigma_n(\mathbf x)(1)$,
and therefore $w_m$ is a proper prefix of $w_n$.
It follows that the limit of $(w_n)$ exists.


For example, using vector $\bx=\left(1, e, \pi\right)^T$, we have
\[
\sigma(\bx) \sigma(F_C\bx) \sigma(F_C^2\bx) \sigma(F_C^3\bx) \sigma(F_C^4\bx)
    = c_{2} c_{1} c_{2} c_{1} c_{1} \\
    = \begin{cases}
	1 \mapsto 23\\
	2 \mapsto 23213\\
	3 \mapsto 2313
    \end{cases},
\]
whose incidence matrix is $M_5(\bx)$. 
The associated infinite $\mathcal{C}$-adic word is
\[
    W(\bx) 
    = 2323213232323132323213232321323231323232 \cdots.
\]

%

\begin{lemma}
Let $\mathbf x \in \Delta$.
The following conditions are equivalent.
\begin{enumerate}
\item[(i)] the entries of $\mathbf x$ are rationally independent,
\item[(ii)] the directive sequence of $W(\mathbf x)$ is primitive,
\item[(iii)] the directive sequence of $W(\mathbf x)$ does not belong
to $\mathcal{C}^* \{c_1^2, c_2^2\}^\omega$.
\end{enumerate}
Furthermore, the vector of letter frequencies of $1$, $2$ and $3$ in $W(\bx)$ is $\bx$.
\end{lemma}

\begin{proof}
Let us first prove that (ii) and (iii) are equivalent.
Assume that $\mathbf s = (\tau_n) \in \mathcal{C}^* \{c_1^2, c_2^2\}^\omega$.
Then there exists $r\in\mathbb N$ such that for all $i\in\mathbb N$,
$\tau_{r+2i} = \tau_{r+2i+1}$, and $\tau_{r+2i}\tau_{r+2i+1}$ is
either $c_1^2$ or $c_2^2$. 
Observe that $c_1^2(1) = 1$, $c_2^2(3)=3$,
and $c_1^2(3) = c_2^2(1) = 13$. Let $r'>r$. 
If $r'-r$ is even,
then $\tau_r\tau_{r+1}\ldots\tau_{r'-1}(1)$ does not contain the letter $2$.
If $r'-r$ is odd,
then $\tau_r\tau_{r+1}\ldots\tau_{r'-1}(2)$ does not contain the letter $2$.
Therefore the directive sequence $\mathbf s$ is not primitive.

Conversely, if $\mathbf{s} \not\in \mathcal{C}^* \{c_1^2, c_2^2\}^\omega$,
then $\mathbf{s}$ contains infinitely many occurrences of words in
$\{c_1c_2^{2i+1}c_1^jc_2^kc_1^lc_2^m,
   c_2c_1^{2i+1}c_2^jc_1^kc_2^lc_1^m
\colon i \in \mathbb N, j,k,l,m \in \mathbb N \setminus \{0\}\}$.
It can be checked that all the matrices of these substitutions have
positive entries, so that $\mathbf s$ is primitive.

Let us now assume that (iii) does not hold. Then, as above,
there exists $r\in\mathbb N$ such that for all $i\in\mathbb N$,
$\tau_{r+2i} = \tau_{r+2i+1}$. Note that, if $\mathbf y=(y_1,y_2,y_3)$,
then $C_1^{-2}\mathbf y=(y_1-y_2-y_3,y_2,y_3)$ and
$C_2^{-2}\mathbf y=(y_1,y_2,y_3-y_1-y_2)$. In both cases, the middle
entry is unchanged, and the sum of the two other entries
decreases by at least $y_2$. Let $F_C^r(\mathbf x) = (y_1,y_2,y_3)$
and $F_C^{r+2i}(\mathbf x) = (z_1,z_2,z_3)$. Then $z_2=y_2$ and
$z_1+z_3 \leq y_1+y_3-iy_2$. This is possible for all $i$ only
if $y_2 = 0$, and then $\ell' F_C^r(\mathbf x) = 0$,
where $\ell'$ is the row vector $\ell' = (0, 1, 0)$.
Then $\ell \mathbf x = 0$ where
$\ell = \ell' M(F_C^{r-1}(\mathbf x))^{-1}\ldots M(\mathbf x)^{-1}$
is a nonzero integer row vector, showing that the entries of $\mathbf x$
are rationally dependent.

Finally, let us assume that (iii) holds and (i) does not hold.
Observe first that, if $F_C^r(\mathbf x)$ has a zero entry
for some $r$, then either $F_C^r(\mathbf x)$ or $F_C^{r+1}(\mathbf x)$
has a zero middle entry, and from this point on the directive sequence
can be factored over $\{c_1^2, c_2^2\}$, contradicting (iii).
From now on we assume that all entries of $F_C^n(\mathbf x)$
are positive for all $n$

Let $\ell_0$
be a nonzero integer row vector such that $\ell_0 \mathbf x = 0$.
The directive sequence can be factored over
$\{c_1c_2^kc_1, c_2c_1^kc_2 \colon k\in\mathbb N\}$. Let us consider the
sequence $(n_m)$ such that $n_0=0$ and $\tau_{n_m}\ldots\tau_{n_{m+1}-1}$
is in this set for all $m\in\mathbb N$.
Let $\ell_m = \ell_0 M(\mathbf x) \ldots M(F_C^{n_m-1}(\mathbf x))$.
Then $\ell_m$ is a nonzero integer row vector such that
$\ell_m F_C^{n_m}(\mathbf x) = 0$, and 
$\ell_{m+1}$ is either $\ell_m C_1C_2^kC_1$ or $\ell_m C_2C_1^kC_2$
for some~$k$.

Assume that $\ell_m = (a, b, c)$. Then $\ell_{m+1}$ is one of
\begin{eqnarray*}
\ell_m C_1C_2^{2k}C_1   &=& (a', a'+b, a'+c) \text{ with } a'=a+kb,\\
\ell_m C_1C_2^{2k+1}C_1 &=& (a'-b, a', a'-c) \text{ with } a'=a+(k+1)b+c,\\
\ell_m C_2C_1^{2k}C_2   &=& (c'+a, c'+b, c') \text{ with } c'=c+kb,\\
\ell_m C_2C_1^{2k+1}C_2 &=& (c'-a, c', c'-b) \text{ with } c'=c+(k+1)b+a.
\end{eqnarray*}

Define $D_m$ as the difference between the maximum and the minimum
entry of $\ell_m$. Note that, as $F_C^{n_m}(\mathbf x)$ has positive entries,
the maximum entry of $\ell_m$ is positive and the minimum entry is negative.
Then $D_m = \max(|b-a|, |c-b|, |c-a|)$.
In the first two cases $D_{m+1} = \max(|b|, |c|, |c-b|)$. If $a$ is
(inclusively) between $b$ and $c$, which must then have opposite signs, then
$D_{m+1} = D_m = |c-b|$. Otherwise $D_{m+1} < D_m$.
Similarly, in the other two cases $D_{m+1} = \max(|a|, |b|, |b-a|)$,
and $D_{m+1} = D_m$ if $c$ is inclusively between $a$ and $b$, 
while $D_{m+1} < D_m$ otherwise.

The sequence of positive integers $(D_m)$ is non-increasing.
To reach a contradiction, we need to show that it decreases
infinitely often.

If for large enough $m$ all transitions between
$\ell_m$ and $\ell_{m+1}$
are of the first type, then (iii) is not satisfied. Similarly,
if for large enough $m$ all transitions are of the third type,
then (iii) is not satisfied. So we must either have infinitely often
transitions of the second or fourth type, or infinitely often
a transition of the first type followed by a transition of the third type.

Assume first that the transition between $\ell_m$
and $\ell_{m+1}$ is of the second type.
Then $\ell_{m+1} = (a'-b, a', a'-c)$ and $\ell_{m+2}$ is one of
\begin{eqnarray*}
\ell_{m+1} C_1C_2^{2k'}C_1   &=& (a'', a''+a', a''+a'-c) \text{ with } a''=a'-b+k'a',\\
\ell_{m+1} C_1C_2^{2k'+1}C_1 &=& (a''-a', a'', a''-a'+c) \text{ with } a''=a'-b+(k'+1)a'+a'-c,\\
\ell_{m+1} C_2C_1^{2k'}C_2   &=& (c''+a'-b, c''+a', c'') \text{ with } c''=a'-c+k'a',\\
\ell_{m+1} C_2C_1^{2k'+1}C_2 &=& (c''-a'+c, c'', c''-a') \text{ with } c''=a'-c+(k'+1)a'+a'-b.
\end{eqnarray*}

If $D_{m+1}=D_m$, then $a$ is between $b$ and $c$ which must have opposite signs.
Then $a'$ is strictly between $a'-b$ and $a'-c$, which implies in all four
cases that $D_{m+2}<D_{m+1}$. So we always have $D_{m+2}<D_m$.

The case where the transition between $\ell_m$ and
$\ell_{m+1}$ is of the fourth type is similar.
Assume now that this transition is of the first type, and the
transition between $\ell_{m+1}$ and $\ell_{m+2}$ is of the third type.
Then $\ell_{m+1} = (a', a'+b, a'+c)$ and
$\ell_{m+2} = (c''+a', c''+a'+b, c'')$ with $c''=a'+c+k'(a'+b)$.
If $D_{m+1}=D_m$, then $a$ is between $b$ and $c$ which must have opposite signs, so that $a'$ is strictly between $a'+b$ and $a'+c$,
which implies that $D_{m+2}<D_{m+1}$. So again we always have $D_{m+2}<D_m$,
and this concludes the proof.
\end{proof}



\section{Factor Complexity of Primitive $\mathcal{C}$-adic Words}

Let $\bw$ be a (infinite) word over some alphabet $A$.
We let $\fac(\bw)$ denote the set of factors of $\bw$, i.e., 
$\fac(\bw) = \{u \in A^* \mid \exists i \in \N: \bw_i \cdots \bw_{i+|u|-1} = u\}$.
The {\em extension set} of $u \in \fac(\bw)$ is the set
$E(u,\bw)=\{(a,b) \in A \times A \mid aub \in \fac(\bw)\}$. 
We represent it by an array of the form
\[
E(u,\bw)\quad=\quad
\begin{array}{c|ccc}
    & \cdots	 	&  j 	& \cdots			\\[-1pt]
\hline
    \cdots	&  &       &       		\\[-1pt]
i	& &		\times					\\[-1pt]
\cdots	&        &			& 			
\end{array},
\]
where a symbol $\times$ in position $(i,j)$ means that $(i,j)$ belongs to $E(u,\bw)$.
When the context is clear we omit the information on $\bw$ and simply write $E(u)$.
We also represent it as an undirected bipartite graph, called the {\em extension graph}, whose set of vertices is the disjoint union of $\pi_1(E(u,\bw))$ and $\pi_2(E(u,\bw))$ ($\pi_1$ and $\pi_2$ respectively being the projection on the first and on the second component) and its edges are the pairs $(a,b) \in E(u,\bw)$.
A factor $u$ of $\bw$ is said to be {\em bispecial} whenever $\# \pi_1(E(u,\bw))>1$ and $\# \pi_2(E(u,\bw))>1$.
A bispecial factor $u \in \fac(\bw)$ is said to be {\em ordinary} if there exists $(a,b) \in E(u,\bw)$ such that $E(u,\bw) \subset (\{a\} \times A) \cup (A \times \{b\})$.


To simplify proofs, we consider 
$\mathcal{C}' = \{c_{11},c_{22},c_{122},c_{211},c_{121},c_{212}\}$, where
\[
\begin{array}{lll}
c_{11} = c_1^2:
\begin{cases}
1 \mapsto 1		\\
2 \mapsto 12 	\\
3 \mapsto 13
\end{cases}
&
c_{122} = c_1 c_2^2:
\begin{cases}
1 \mapsto 12	\\
2 \mapsto 132 	\\
3 \mapsto 2
\end{cases}
&
c_{121} = c_1c_2c_1:
\begin{cases}
1 \mapsto 13		\\
2 \mapsto 132 	\\
3 \mapsto 12
\end{cases}
\\ 
\\
c_{22} = c_2^2:
\begin{cases}
1 \mapsto 13	\\
2 \mapsto 23 	\\
3 \mapsto 3
\end{cases}
&
c_{211} = c_2 c_1^2:
\begin{cases}
1 \mapsto 2	\\
2 \mapsto 213 	\\
3 \mapsto 23
\end{cases}
&
c_{212} = c_2c_1c_2:
\begin{cases}
1 \mapsto 23		\\
2 \mapsto 213 	\\
3 \mapsto 13
\end{cases}
\end{array}.
\]

Any (primitive) $\mathcal{C}$-adic word is a (primitive) $\mathcal{C}'$-adic word and conversely. We let $\varepsilon$ denote the empty word.
We have the following result, where uniqueness follows
from the fact that $\tau(A)$ is a code.

\begin{lemma}[Synchronization]
\label{lemma:synchronization}
Let $\bw$ be a $\mathcal{C}'$-adic word with directive sequence $(\tau_n)_{n \in \N} \in \mathcal{C}'^\N$.
If $u \in \fac(\bw)$ is a non-empty bispecial factor, then
\begin{enumerate}
\item
If $\tau_0 = c_{11}$, there is a unique word $v \in \fac(\bw^{(1)})$ such that $u = \tau_0(v) 1$.
\item
If $\tau_0 = c_{22}$, there is a unique word $v \in \fac(\bw^{(1)})$ such that $u = 3 \tau_0(v)$.
\item
If $\tau_0 = c_{122}$, there is a unique word $v \in \fac(\bw^{(1)})$ such that $u \in 2 \tau_0(v) \{1,\varepsilon\}$.
\item
If $\tau_0 = c_{211}$, there is a unique word $v \in \fac(\bw^{(1)})$ such that $u \in \{3,\varepsilon\} \tau_0(v) 2$.
\item
If $\tau_0 = c_{121}$, there is a unique word $v \in \fac(\bw^{(1)})$ such that $u \in \{2,\varepsilon\} \tau_0(v) \{1,13\}$.
\item
If $\tau_0 = c_{212}$, there is a unique word $v \in \fac(\bw^{(1)})$ such that $u \in \{3,13\} \tau_0(v) \{2,\varepsilon\}$.
\end{enumerate}
Furthermore, $v$ is a bispecial factor of $\bw^{(1)}$ and is shorter than $u$.
\end{lemma}

Let $\bw$, $u$ and $v$ be as in Lemma~\ref{lemma:synchronization}. 
The word $v$ is called the {\em bispecial antecedent of $u$ under $\tau_0$}. 
Similarly, $u$ is called a {\em bispecial extended image of $v$ under $\tau_0$}.
Since the bispecial antecedent of a non-empty bispecial word is always shorter, for any bispecial factor $u$ of $\bw$, there is a unique sequence $(u_i)_{0 \leq i \leq n}$ such that 
\begin{itemize}
\item
$u_0 = u$, $u_n = \varepsilon$ and $u_i \neq \varepsilon$ for all $i < n$;
\item
for all $i < n$, $u_{i+1} \in \fac(\bw^{(i+1)})$ is the bispecial antecedent of $u_i$.
\end{itemize}
All bispecial factors of the sequence $(u_i)_{0 \leq i < n}$ are called the {\em bispecial descendants} of $\varepsilon$ in $\bw^{(n)}$.

As any bispecial factor of a primitive $\mathcal{C}$-adic word is a descendant of the empty word, to understand the extension sets of any bispecial word in $\bw$, we need to know the possible extension sets of $\varepsilon$ in $\bw^{(n)}$ and to understand how the extension set of a bispecial factor governs the extension sets of its bispecial extended images.

\begin{lemma}
\label{lemma: empty word}
If $\bw$ is a primitive $\mathcal{C}$-adic word with directive sequence $(\tau_n)_{n \in \N} \in \mathcal{C}'^\N$, then the extension set of $\varepsilon$ is one of the following, depending on $\tau_0$.

\medskip

\begin{center}
\begin{tabular}{ccc}
 \begin{tabular}{c|ccc}
 $\tau_0 = c_{11}$ 
 & $1$ & $2$ & $3$ \\
 \hline
 $1$ & $\times$ & $\times$ & $\times$ \\
 $2$ & $\times$ &   &   \\
 $3$ & $\times$ &   &   \\
 \end{tabular}
 &\qquad
 \begin{tabular}{c|ccc}
 $\tau_0 = c_{122}$ 
 & $1$ & $2$ & $3$ \\
 \hline
 $1$ &   & $\times$ & $\times$ \\
 $2$ & $\times$ & $\times$ &   \\
 $3$ &   & $\times$ &   
 \end{tabular}
 &\qquad
 \begin{tabular}{c|ccc}
 $\tau_0 = c_{121}$ 
 & $1$ & $2$ & $3$ \\
 \hline
 $1$ &   & $\times$ & $\times$ \\
 $2$ & $\times$ &   &   \\
 $3$ & $\times$ & $\times$ &   \\
 \end{tabular}
    \\ \\
 \begin{tabular}{c|ccc}
 $\tau_0 = c_{22}$ 
 & $1$ & $2$ & $3$ \\
 \hline
 $1$ &   &   & $\times$ \\
 $2$ &   &   & $\times$ \\
 $3$ & $\times$ & $\times$ & $\times$ \\
 \end{tabular}
 &\qquad
 \begin{tabular}{c|ccc}
 $\tau_0 = c_{211}$ 
 & $1$ & $2$ & $3$ \\
 \hline
 $1$ &   &   & $\times$ \\
 $2$ & $\times$ & $\times$ & $\times$ \\
 $3$ &   & $\times$ &   \\
 \end{tabular}
 &\qquad
 \begin{tabular}{c|ccc}
 $\tau_0 = c_{212}$ 
 & $1$ & $2$ & $3$ \\
 \hline
 $1$ &   &   & $\times$ \\
 $2$ & $\times$ &   & $\times$ \\
 $3$ & $\times$ & $\times$ &   \\
 \end{tabular}
 \end{tabular}
 \end{center}
\end{lemma}
\begin{proof}
The directive sequence being primitive, all letters of $\cal A$ occur in $\bw^{(1)}$.
The result then follows from the fact that all morphisms $\tau$ in $\mathcal{C}'$ are either left proper ($\tau(\mathcal{A}) \subset a \mathcal{A}^*$ for some letter $a$) or right proper ($\tau(\mathcal{A}) \subset \mathcal{A}^* a$ for some letter $a$). 
\end{proof}

The next lemma describes how the extension set of a bispecial word determines the extension set of any of its bispecial extended images.

\begin{lemma}
\label{lemma:extensions from antecedent}
Let $\bw$ be a $\mathcal{C}'$-adic word with directive sequence $(\tau_n)_{n \in \N} \in \mathcal{C}'^\N$.
If $u \in \fac(\bw)$ is the bispecial extended image of $v \in \fac(\bw^{(1)})$  and if $x,y  \in \mathcal{A}^*$ are such that $u = x \tau_0(v) y$, then
\begin{enumerate}
\item
if $\tau_0(\mathcal{A}) \subset i \mathcal{A}^*$ for some letter $i \in \mathcal{A}$, we have
\[
	E(u,\bw) = \{(a,b) \mid	\exists (a',b') \in E(v,\bw^{(1)}): 
	 \tau_0(a') \in \mathcal{A}^*ax \, \wedge \, \tau_0(b')i \in y b \mathcal{A}^*\};
\]
\item
if $\tau_0(\mathcal{A}) \subset \mathcal{A}^* i$ for some letter $i \in \mathcal{A}$, we have
\[
	E(u,\bw) = \{(a,b) \mid	\exists (a',b') \in E(v,\bw^{(1)}): 
	 i \tau_0(a') \in \mathcal{A}^*ax \, \wedge \, \tau_0(b') \in y b \mathcal{A}^*\}.
\]
\end{enumerate}
\end{lemma}
\begin{proof}
Let us prove the first equality, the second one being symmetric.

For the inclusion $\supseteq$, consider $(a',b') \in E(v)$ such that $\tau_0(a') \in \mathcal{A}^* ax$ and $\tau_0(b')i \in y b \mathcal{A}^*$.
Let $c\in\cal A$ be such that $a'vb'c$ is a factor of $\mathbf w^{(1)}$.
Then $\tau_0(a'vb'c) \in \tau_0(a'vb')i \mathcal{A}^* \subseteq \mathcal{A}^* ax \tau_0(v) yb \mathcal{A}^*$
is a factor of $w$ and we have $(a,b) \in E(u)$.

For the inclusion $\subseteq$, consider $(a,b) \in E(u)$.
Using Lemma~\ref{lemma:synchronization}, the word $ax$ (resp., $yb$) is the suffix (resp., prefix) of a word $\tau_0(x')$, $x' \in \mathcal{A}^+$ (resp., $\tau_0(y')$, $y' \in \mathcal{A}^+$) such that $x' v y' \in \fac(\bw^{(1)})$.
Furthermore, still using Lemma~\ref{lemma:synchronization}, $x$ is a strict suffix of $\tau_0(a')$, where $x' \in \mathcal{A}^* a'$ and $y$ is a prefix of $\tau_0(b')$, where $y' \in b' \mathcal{A}^*$.
If $y$ is a strict prefix of $\tau_0(b')$, then $(a',b')$ is an extension of $v$ such that $\tau_0(a') \in \mathcal{A}^* ax$ and $\tau_0(b') \in yb \mathcal{A}^*$.
Otherwise, if $\tau_0(b')=y$, we have $b = i$ since $\tau_0(\mathcal{A}) \subset i \mathcal{A}^*$ and $(a',b')$ is an extension of $v$ such that $\tau_0(a') \in \mathcal{A}^* ax$ and $\tau_0(b')i = yi$, which concludes the proof.
\end{proof}

Lemma~\ref{lemma:extensions from antecedent} can be more easily understood using the tabular representation of the extension sets.
Indeed, for the first case ($\tau_0(\mathcal{A}) \subset i \mathcal{A}^*$), the extensions of $u = x \tau(v) y$ can be obtained as follows: 1) replace any left extensions $a$ by $\tau(a)$ and any right extension $b$ by $\tau(b)i$; 2) remove the suffix $x$ from the left extensions whenever it is possible (otherwise, delete the row) and remove the prefix $y$ from the right extensions whenever it is possible (otherwise, delete the column); 3) keep only the last letter of the left extensions and the first letter of the right extensions; 4) permute and merge the rows and columns with the same label. 
The second case ($\tau_0(\mathcal{A}) \subset \mathcal{A}^* i$) is similar.


Let us make this more clear on an example and consider the extension set $E(v) = \{(1,3),(2,1),(2,2),(2,3),(3,2)\}$.
This extension set corresponds to the extension set of the empty word whenever the last applied substitution is $c_{211}$ (see Lemma~\ref{lemma: empty word}).
Using Lemma~\ref{lemma:extensions from antecedent}, the extension sets of $2 c_{122}(v)$ and $2 c_{121}(v)1$ are obtained as follows (arrow labels indicate above step number):
\begin{eqnarray*}
\begin{array}{c}
E(v)\\
\begin{array}{r|ccc}
  & 1 & 2 & 3	\\
\hline
 1 &   		&   	 & \times 	\\
 2 & \times & \times & \times 	\\
 3 &   		& \times &   
\end{array}
\end{array}
\xrightarrow[]{1)}
\begin{array}{c}
E(c_{122}(v))\\
\begin{array}{r|ccc}
 & 12 & 132 & 2	\\
\hline
 212 	&   		&   	 & \times 	\\
 2132 	& \times 	& \times & \times 	\\
 22 	&   		& \times &   
\end{array}
\end{array}
& \xrightarrow[]{2) \text{ and }3)} &
\begin{array}{c}
E(2c_{122}(v))\\
\begin{array}{r|ccc}
& 1 & 1 & 2	\\
\hline
 1 	&   		&   	 & \times 	\\
 3 	& \times 	& \times & \times 	\\
 2	&   		& \times &   
\end{array}
\end{array}
\xrightarrow[]{4)}
\begin{array}{c}
E(2c_{122}(v))\\
\begin{array}{r|cc}
& 1 & 2 \\
\hline
 1 & 		& \times 		\\
 2 & \times &   			\\
 3 & \times	& \times    
\end{array}
\end{array}
\\
\\
\begin{array}{c}
E(v)\\
\begin{array}{r|ccc}
& 1 & 2 & 3	\\
\hline
 1 &   		&   	 & \times 	\\
 2 & \times & \times & \times 	\\
 3 &   		& \times &   
\end{array}
\end{array}
\xrightarrow[]{1)}
\begin{array}{c}
E(c_{121}(v))\\
\begin{array}{r|ccc}
& 131 & 1321 & 121	\\
\hline
 13 	&   		&   	 & \times 	\\
 132 	& \times 	& \times & \times 	\\
 12 	&   		& \times &   
\end{array}
\end{array}
& \xrightarrow[]{2) \text{ and }3)} &
\begin{array}{c}
E(2c_{121}(v)1)\\
\begin{array}{r|ccc}
& 3 & 3 & 2	\\
\hline
 .	&   		&        &   \\
 3 	& \times 	& \times & \times 	\\
 1	&   		& \times &   
\end{array}
\end{array}
\xrightarrow[]{4)}
\begin{array}{c}
E(2c_{121}(v)1)\\
\begin{array}{r|cc}
& 2 & 3 \\
\hline
 1 & 		& \times 		\\
 3 & \times	& \times    
\end{array}
\end{array}
\end{eqnarray*}

The proof of Proposition~\ref{prop: cassaigne are tree} will essentially consists in describing how ordinary bispecial words occur.
The next lemma allows to understand when bispecial words have ordinary bispecial extended images.

\begin{lemma}
\label{lemma:ordinary preserved}
Let $\bw$ be a $\mathcal{C}'$-adic word with directive sequence $(\tau_n)_{n \in \N} \in \mathcal{C}'^\N$.
Let $u \in \fac(\bw)$ be a non-empty bispecial factor and $v$ be its bispecial antecedent. We have the following.
\begin{enumerate}
\item \label{item:E(v)=E(u)}
If $\tau_0 \in \{c_{11},c_{22}\}$, then $E(u)=E(v)$;
\item \label{item:v=epsilon}
if $v = \varepsilon$ and $\tau_0 \in \{c_{121},c_{212}\}$, then $u$ is ordinary;
\item \label{item:E(v) C12}
if $ \tau_0 \in \{c_{122},c_{121},c_{212}\}$, if $E(v) \subseteq (A \times \{1,2\}) \cup \{(a,3)\}$ for some letter $a \in A$ with $E(v) \cap \{(a,1), (a,2)\} \neq\emptyset$ and if $E(v) \setminus \{(a,3)\}$ is the extension set of an ordinary bispecial word, then $u$ is ordinary;
\item \label{item:E(v) L23}
if $ \tau_0 \in \{c_{211},c_{121},c_{212}\}$, if $E(v) \subseteq (\{2,3\} \times A) \cup \{(1,a)\}$ for some letter $a \in A$ with $E(v) \cap \{(2,a), (3,a)\} \neq\emptyset$ and if $E(v) \setminus \{(1,a)\}$ is the extension set of an ordinary bispecial word, then $u$ is ordinary;
\item \label{item:ordinary preserved}
if $v$ is ordinary, then $u$ is ordinary
\end{enumerate}
\end{lemma}
\begin{proof}
Items~\ref{item:E(v)=E(u)} and~\ref{item:ordinary preserved} directly follow from Lemma~\ref{lemma:extensions from antecedent}. Item~\ref{item:v=epsilon} can be checked by hand using Lemma~\ref{lemma: empty word} and Lemma~\ref{lemma:extensions from antecedent}.
Let us prove Item~\ref{item:E(v) C12}, Item~\ref{item:E(v) L23} being symmetric.

Let us say that two extension sets $E$ and $E'$ are equivalent whenever there exist two permutations $p_1$ and $p_2$ of $A$ such that $E = \{(p_1(a),p_2(b)) \mid (a,b) \in E'\}$.
If $\tau_0 = c_{122}$, then $u \in \{2 \sigma(v), 2\sigma(v)1\}$ by Lemma~\ref{lemma:synchronization}. 
We make use of Lemma~\ref{lemma:extensions from antecedent}. 
If $u = 2 \sigma(v)$, then the extension set of $u$ is equivalent to the one obtained from $E(v)$ by merging the columns with labels 1 and 2. 
If $u = 2 \sigma(v)1$, then the extension set of $u$ is equivalent to the one obtained from $E(v)$ by deleting the column with label 3.
In both cases, $u$ is ordinary.

The same reasoning applies when $\tau_0 \in \{c_{121},c_{212}\}$: depending on the word $x$ such that $u \in A^* \sigma(v)x$, either we delete the column with label 3, or we merge the columns with labels 1 and 2.   
\end{proof}

Recall that an infinite word is a {\em tree word} if the extension graph of any  of its bispecial factors is a tree.
Obviously, if $u$ is an ordinary bispecial word, its extension graph is a tree.
If $\bw \in \mathcal{A}^\N$ is a tree word in which all letters of $\mathcal{A}$ occur, then $\bw$ has factor complexity $p(n) = (\card(\mathcal{A})-1)n+1$ for all $n$~\cite{MR3320917}.

\begin{proposition}
\label{prop: cassaigne are tree}
Any primitive $\mathcal{C}$-adic word is a uniformly recurrent tree word. 
In particular, any primitive $\mathcal{C}$-adic word has factor complexity $p(n)=2n+1$.
\end{proposition}

\begin{proof}
Any primitive $\mathcal{C}$-adic word has a primitive $\mathcal{C}$-adic representation, hence is uniformly recurrent.

To show that the extension graphs of all bispecial factors are trees, we make use of Lemma~\ref{lemma:ordinary preserved}.
If $u$ is a bispecial factor of $\bw$, it is a descendant of $\varepsilon \in \fac(\bw^{(n)})$ for some $n$.
If $\tau_n \in \{c_{11},c_{22}\}$, then from Lemma~\ref{lemma: empty word} and Lemma~\ref{lemma:ordinary preserved}, all descendants of $\varepsilon$ are ordinary. 
The extension graph of $u$ is thus a tree.

For $\tau_n \in\{c_{122},c_{211},c_{121},c_{212}\}$, we represent the extension sets of the descendants of $\varepsilon$ in the graphs represented in Figure~\ref{figure:descendants of c_{122}} and Figure~\ref{figure:descendants of c_{121}}.
Observe that the situation is symmetric for $c_{122}$ and $c_{211}$ and for $c_{121}$ and $c_{212}$ so we only represent the graphs for $c_{122}$ and $c_{121}$.
Furthermore, in these graphs, we do not represent the extension sets of ordinary bispecial factors as the property of being ordinary is preserved by taking bispecial extended images (Lemma~\ref{lemma:ordinary preserved}).
Given an extension set of some bispecial word $v$, if $u$ is a bispecial extended image of $v$ such that $u = x \tau(v) y$, we label the edge from $E(v)$ to $E(u)$ by $x \cdot \tau \cdot y$.
Finally, for all $v$, we have $E(c_{11}(v)1)=E(v)$ and $E(3c_{22}(v))=E(v)$, but for the sake of clarity, we do not draw the loops labeled by $c_{11}\cdot 1$ and by $3 \cdot c_{22}$.
We conclude the proof by observing that the extension graphs of all descendants are trees. 
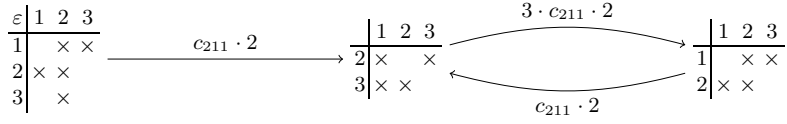
\begin{figure}[t]
\centering
\scalebox{.9}{
\begin{tikzpicture}
\node(A) at (0,0){
$\begin{array}{c|ccc}
\varepsilon & 1 & 2 & 3 \\
 \hline
 1 &   		& \times & \times \\
 2 & \times & \times &   \\
 3 &   		& \times &  
\end{array}$
};
\node(B) at (5,0){
$\begin{array}{c|ccc}
 & 1 & 2 & 3 \\
 \hline
 2 & \times & 		 & \times  \\
 3 & \times	& \times &  
\end{array}$
};
\node(C) at (10,0){
$\begin{array}{c|ccc}
 & 1 & 2 & 3 \\
 \hline
 1 & 		& \times & \times  \\
 2 & \times	& \times &  
\end{array}$
};
\path [->] (A) edge node[above] {$c_{211}\cdot 2$} (B);
\path [->] (B) edge[bend left=15] node[above] {$3\cdot c_{211}\cdot 2$} (C);
\path [->] (C) edge[bend left=15] node[below] {$c_{211}\cdot 2$} (B);
\end{tikzpicture}
}
\caption{Non-ordinary bispecial descendants of $\varepsilon \in \fac(\bw^{(n)})$ whenever $\tau_n=c_{122}$.}
\label{figure:descendants of c_{122}}
\end{figure}
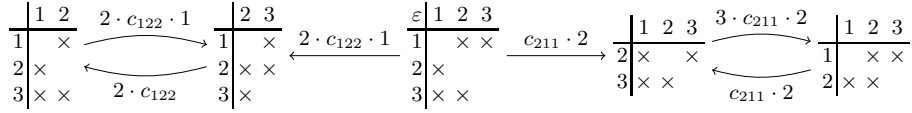
\begin{figure}[t]
\centering
\scalebox{.9}{
\begin{tikzpicture}
\node(A) at (6,0){
$\begin{array}{c|ccc}
\varepsilon & 1 & 2 & 3 \\
 \hline
 1 &   		& \times 	& \times 	\\
 2 & \times &   		&   		\\
 3 & \times & \times 	&
\end{array}$
};
\node(B) at (3,0){
$\begin{array}{c|cc}
 & 2 & 3 \\
 \hline
 1 & 		& \times  \\
 2 & \times & \times  \\
 3 & \times	&   
\end{array}$
};
\node(B') at (0,0){
$\begin{array}{c|cc}
 & 1 & 2 \\
 \hline
 1 & 		& \times  \\
 2 & \times & 		  \\
 3 & \times	& \times  
\end{array}$
};
\node(C) at (9,0){
$\begin{array}{c|ccc}
 & 1 & 2 & 3 \\
 \hline
 2 & \times &		 & \times  \\
 3 & \times	& \times &  
\end{array}$
};
\node(C') at (12,0){
$\begin{array}{c|ccc}
 & 1 & 2 & 3 \\
 \hline
 1 & 		& \times & \times  \\
 2 & \times	& \times &  
\end{array}$
};
\path [->] (A) edge node[above] {$2\cdot c_{122}\cdot 1$} (B);
\path [->] (B) edge[bend left=15] node[below] {$2\cdot c_{122}$} (B');
\path [->] (B') edge[bend left=15] node[above] {$2 \cdot c_{122} \cdot 1$} (B);
\path [->] (A) edge node[above] {$c_{211} \cdot 2$} (C);
\path [->] (C) edge[bend left=15] node[above] {$3 \cdot c_{211} \cdot 2$} (C');
\path [->] (C') edge[bend left=15] node[below] {$c_{211} \cdot 2$} (C);
\end{tikzpicture}
}
\caption{Non-ordinary bispecial descendants of $\varepsilon \in \fac(\bw^{(n)})$ whenever $\tau_n=c_{121}$.}
\label{figure:descendants of c_{121}}
\end{figure}

\end{proof}

\section{Selmer Algorithm}

Selmer algorithm \cite{schweiger} (also called \emph{the GMA algorithm} in
\cite{MR1156412}) is an algorithm which subtracts the smallest entry to the
largest. Here we introduce a semi-sorted version of it which keeps the largest
entry at index $1$. On $\Gamma=\{\bx=(x_1,x_2,x_3)\in\mathbb{R}^3_{\geq 0}\mid
\max(x_2,x_3)\leq x_1\leq x_2+x_3\}$, it is defined as
\[
    F_S (x_1,x_2,x_3) = 
\begin{cases}
    (x_2, x_1-x_3, x_3) & \mbox{if } x_2 \geq x_3,\\
    (x_3, x_2, x_1-x_2) & \mbox{if } x_2 < x_3.
\end{cases}
\]
The partition of $\Gamma$ into $\Gamma_1\cup\Gamma_2$ is
\begin{align*}
	\Gamma_1 &= \{(x_1,x_2,x_3)\in\Gamma\mid 
	x_2 \geq x_3\}, \\
    \Gamma_2 &= \{(x_1,x_2,x_3)\in\Gamma\mid 
	x_2 < x_3\}.
\end{align*}
For semi-sorted Selmer algorithm, the matrices and associated
substitutions are
\[
S_1 = \left(\begin{array}{rrr}
    0 & 1 & 1 \\
    1 & 0 & 0 \\
    0 & 0 & 1
\end{array}\right),
\quad
S_2 = \left(\begin{array}{rrr}
    0 & 1 & 1 \\
    0 & 1 & 0 \\
    1 & 0 & 0
\end{array}\right)
\quad\text{and}\quad
s_1=\left\{\begin{array}{l}
1 \mapsto 2\\
2 \mapsto 1\\
3 \mapsto 31
\end{array}\right.,
\quad
s_2=\left\{\begin{array}{l}
1 \mapsto 3\\
2 \mapsto 12\\
3 \mapsto 1
\end{array}\right..
\]
The matrices are given by the rule
$M(\bx) = S_i$
if and only if
$\bx\in\Gamma_i$.
The map $F_S$ on $\Gamma$ 
is then defined as:
$
    F_S(\bx) = M(\bx)^{-1}\bx$.
The substitutions on $\A^*$ are given by the rule
$\sigma(\bx) = s_i$
if and only if
$\bx\in\Gamma_i$ for $i=1,2$.

\section{Conjugacy of $F_C$ and $F_S$}

The numerical computation of Lyapunov exponents made in \cite{labbe_3-dimensional_2015}
indicate that exponents for the unsorted Selmer algorithm and $F_C$ have
statistically equal values. The next proposition gives an explanation for this observation.

\begin{proposition}
    Algorithms $F_C:\Lambda\to\Lambda$ and $F_S:\Gamma\to\Gamma$ are topologically conjugate.
\end{proposition}

\begin{proof}
Let $z:\Lambda\to\Gamma$ be the homeomorphism defined by $\bx\mapsto Z\bx$ with
\[
Z= \left(\begin{array}{rrr}
    1 & 1 & 1 \\
    1 & 1 & 0 \\
    0 & 1 & 1
\end{array}\right).
\]
We verify that $C_i$ is conjugate to $S_i$ through matrix $Z$ for $i=1,2$:
\[
S_1 Z=
\left(\begin{array}{rrr}
1 & 2 & 1 \\
1 & 1 & 1 \\
0 & 1 & 1
\end{array}\right)
= Z C_1
\qquad\text{and}\qquad
S_2 Z=
\left(\begin{array}{rrr}
    1 & 2 & 1 \\
    1 & 1 & 0 \\
    1 & 1 & 1
\end{array}\right)
= Z C_2.
\]
Thus we have $z\circ F_C = F_S\circ z$.
\end{proof}

An infinite word $u\in A^\N$ is said to be \emph{finitely balanced} if there
exists a constant $C>0$ such that for any pair $v$, $w$ of factors of the same
length of $u$, and for any letter $i\in A$, $||v|_i-|w|_i|\leq C$.

Based on \cite[Theorem 6.4]{MR3330561}, and considering that
computer experiments suggest that the second Lyapunov exponent of Selmer algorithm
is negative ($\theta_1\approx\log(1.200)\approx 0.182$ and
$\theta_2\approx\log(0.9318)\approx -0.0706$ in \cite[p. 1522]{MR1156412},
$\theta_1\approx 0.18269$ and $\theta_2\approx -0.07072$ in
\cite{labbe_3-dimensional_2015}), we believe that the following conjecture
holds.

\begin{conjecture}
For almost every $\bx\in\Delta$, the word $W(\bx)$ is finitely balanced.
\end{conjecture}

\subsection{Substitutive Conjugacy}
Let $z_l$ and $z_r$ be the following two substitutions:
\[
z_l :\left\{\begin{array}{l}
1 \mapsto 12\\
2 \mapsto 123\\
3 \mapsto 13
\end{array}\right.
\qquad\text{and}\qquad
z_r :\left\{\begin{array}{l}
1 \mapsto 21\\
2 \mapsto 231\\
3 \mapsto 31
\end{array}\right..
\]
The substitution $z_l$ is left proper while $z_r$ is right proper.
Moreover they are conjugate through the equation
\[
    z_l(w)\cdot 1 = 1\cdot z_r(w) \qquad\text{for every } w\in\mathcal{A}^*.
\]
Notice that $Z$ is the incidence matrix of both $z_l$ and $z_r$.


The substitutions $c_i$ are not conjugate to $s_i$ but are related through
substitutions $z_l$ and $z_r$ for $i=1,2$:
\begin{align*}
s_1\circ z_l = z_r\circ c_1 = (1 \mapsto 21, 2 \mapsto 2131, 3 \mapsto 231),\\
s_2\circ z_r = z_l\circ c_2 = (1 \mapsto 123, 2 \mapsto 1213, 3 \mapsto 13).
\end{align*}

We deduce that
\begin{proposition}
$S$-adic sequences when $S=\{s_1,s_2\}$ restricted to the application of
the semi-sorted Selmer algorithm $F_S$ on totally irrational vectors
$\bx\in\Gamma$ have factor complexity $2n+1$.
\end{proposition}


The problem of finding an analogue of $F_C$ in dimension $d\geq 4$ (i.e.
projective dimension $d-1$), generating $S$-adic sequences with complexity
$(d-1)n+1$ for almost every vector of letter frequencies is still open.

\section*{Acknowledgments}

We are thankful to Valérie Berthé for her enthusiasm toward this project and
for the referees for their thorough reading and pertinent suggestions.

\bibliographystyle{splncs03} 
\bibliography{biblio}

\begin{thebibliography}{10}
\providecommand{\url}[1]{\texttt{#1}}
\providecommand{\urlprefix}{URL }

\bibitem{arnoux_labbe_2017}
Arnoux, P., Labbé, S.: On some symmetric multidimensional continued fraction
  algorithms. Ergodic Theory and Dynamical Systems pp. 1--26 (2017),
  \url{http://dx.doi.org/10.1017/etds.2016.112}

\bibitem{MR1116845}
Arnoux, P., Rauzy, G.: Repr\'esentation g\'eom\'etrique de suites de
  complexit\'e {$2n+1$}. Bull. Soc. Math. France  119(2),  199--215 (1991)

\bibitem{MR3184185}
Arnoux, P., Starosta, {\v{S}}.: The {R}auzy gasket. In: Further developments in
  fractals and related fields, pp. 1--23. Trends Math., Birkh\"auser/Springer,
  New York (2013), \url{http://dx.doi.org/10.1007/978-0-8176-8400-6_1}

\bibitem{MR1156412}
Baldwin, P.R.: A convergence exponent for multidimensional continued-fraction
  algorithms. J. Statist. Phys.  66(5-6),  1507--1526 (1992)

\bibitem{MR3283831}
Berth{\'e}, V., Labb{\'e}, S.: Factor complexity of {$S$}-adic words generated
  by the {A}rnoux-{R}auzy-{P}oincar\'e algorithm. Adv. in Appl. Math.  63,
  90--130 (2015), \url{http://dx.doi.org/10.1016/j.aam.2014.11.001}

\bibitem{MR3320917}
Berth\'e, V., De~Felice, C., Dolce, F., Leroy, J., Perrin, D., Reutenauer,
  Rindone, G.: Acyclic, connected and tree sets. Monatsh. Math.  176(4),
  521--550 (2015), \url{http://dx.doi.org/10.1007/s00605-014-0721-4}

\bibitem{MR3330561}
Berth\'e, V., Delecroix, V.: Beyond substitutive dynamical systems: {$S$}-adic
  expansions. In: Numeration and substitution 2012, pp. 81--123. RIMS
  K\^oky\^uroku Bessatsu, B46, Res. Inst. Math. Sci. (RIMS), Kyoto (2014)

\bibitem{BRENTJES}
Brentjes, A.J.: Multidimensional continued fraction algorithms. Mathematisch
  Centrum, Amsterdam (1981)

\bibitem{cassaigne_algorithme_2015}
Cassaigne, J.: Un algorithme de fractions continues de complexité linéaire
  \url{http://www.irif.fr/~dyna3s/Oct2015}, {D}ynA3S, LIAFA, Paris, Oct. 12th,
  2015

\bibitem{MR3124516}
Delecroix, V., Hejda, T., Steiner, W.: Balancedness of {A}rnoux-{R}auzy and
  {B}run words. In: Combinatorics on words, Lecture Notes in Comput. Sci., vol.
  8079, pp. 119--131. Springer (2013),
  \url{http://dx.doi.org/10.1007/978-3-642-40579-2_14}

\bibitem{labbe_3-dimensional_2015}
Labbé, S.: 3-dimensional {Continued} {Fraction} {Algorithms} {Cheat} {Sheets}
  (Nov 2015), \arxiv{1511.08399}

\bibitem{MR3214265}
Leroy, J.: An {$S$}-adic characterization of minimal subshifts with first
  difference of complexity $1\le p(n+1)-p(n)\le2$. Discrete Math. Theor.
  Comput. Sci.  16(1),  233--286 (2014)

\bibitem{schweiger}
Schweiger, F.: Multidimensional Continued Fractions. Oxford Univ. Press, New
  York (2000)

\bibitem{MR1488330}
Zorich, A.: Deviation for interval exchange transformations. Ergodic Theory
  Dynam. Systems  17(6),  1477--1499 (1997)

\end{thebibliography}

\end{document}